\documentclass[11pt]{amsart}

\usepackage{amssymb}
\usepackage{mathrsfs}
\usepackage{enumerate}
\usepackage{esint}
\usepackage{titletoc}
\usepackage[colorlinks=true, urlcolor=blue, linkcolor=blue, citecolor=magenta]{hyperref}
\usepackage[nameinlink]{cleveref}
\usepackage{graphicx}

\theoremstyle{plain}
\newtheorem{theorem}{Theorem}[section]
\newtheorem{lemma}[theorem]{Lemma}

\theoremstyle{definition}

\numberwithin{equation}{section}

\makeatletter
\@namedef{subjclassname@2020}{\textup{2020} Mathematics Subject Classification}
\makeatother

\newcommand{\R}{{\mathbb R}}

\newcommand{\bS}{\mathbb{S}}

\newcommand{\qu}{\quad}
\newcommand{\D}{\nabla}

\newcommand{\fr}{\frac}

\title[Diameter estimate for planar $L_p$ dual Minkowski problem]{Diameter estimate for planar $L_p$ dual Minkowski problem}

\author{Minhyun Kim}
\address{Fakult\"at f\"ur Mathematik, Universit\"at Bielefeld, 33615 Bielefeld, Germany}
\email{minhyun.kim@uni-bielefeld.de}

\author{Taehun Lee}
\address{School of Mathematics, Korea Institute for Advanced Study, Seoul 02455, Korea}
\email{taehun@kias.re.kr}


\begin{document}

\begin{abstract}
In this paper, given a prescribed measure on $\bS^1$ whose density is bounded and positive, we establish a uniform diameter estimate for solutions to the planar $L_p$ dual Minkowski problem when $0<p<1$ and $q\ge 2$. We also prove the uniqueness and positivity of solutions to the $L_p$ Minkowski problem when the density of the measure is sufficiently close to a constant in $C^\alpha$.
\end{abstract}

\maketitle


\section{Introduction} 


The Minkowski problem, posed by Minkowski \cite{Min03}, is one of the central problems in convex geometry. It asks if a given measure on the unit sphere $\mathbb{S}^{n-1}$ arises as the surface area measure of a convex body. Here, a {\it convex body $K$} is a compact convex set of $\mathbb{R}^{n}$ with a nonempty interior, and its surface area measure $S(K, \cdot)$ is determined by the celebrated Aleksandrov variational formula
\begin{equation}
\left. \frac{\mathrm{d}\mathrm{Vol}(K+tL)}{\mathrm{d}t} \right|_{t = 0^{+}} = \int_{\mathbb{S}^{n-1}} h_{L}(u) \,\mathrm{d}S(K, u) \quad \text{for any convex body }L,
\end{equation}
where $\mathrm{Vol}$ denotes the $n$-dimensional volume and $h_L:\mathbb{S}^{n-1}\to \mathbb{R}$ is the {\it support function of $L$} defined by $h_L(u) = \max \lbrace u \cdot x: x \in L \rbrace$.

An important variant of the Minkowski problem is the $L_p$ Minkowski problem initiated by Lutwak \cite{Lut93}. To describe the $L_p$ Minkowski problem, we first recall Firey's \cite{Fir74} $p$-linear combination of convex bodies. Let us denote by $\mathcal{K}^{n}_{o}$ the set of all convex bodies containing the origin (not necessarily in its interior). For $K, L \in \mathcal{K}^{n}_{o}$, $p \geq 1$, and $t, s > 0$, the {\it $p$-linear combination $t \cdot_{p} K +_{p} s \cdot_{p} L$ of $K$ and $L$} is defined as the convex body whose support function is given by $(t h_K^p + s h_L^p)^{1/p}$. Then, there exists a Borel measure $S_p(K,\cdot)$ on $\mathbb{S}^{n-1}$ such that
\begin{equation}\label{eq:S_Kp}
\left. \frac{\mathrm{d}\mathrm{Vol}(K+_{p}t \cdot_{p}L)}{\mathrm{d}t} \right|_{t = 0^{+}} = \frac{1}{p} \int_{\mathbb{S}^{n-1}} h_{L}^p(u) \,\mathrm{d}S_p(K, u)
\end{equation}
for all convex bodies $L$. We call this measure $S_p(K,\cdot)$ the {\it $L_p$ surface area measure}, and it is known \cite{Lut93} that $S_p(K,\cdot)$ is related to $S(K,\cdot)$ by
\begin{equation} \label{eq:Sp-density}
S_p(K,\cdot) = h_K^{1-p} S(K,\cdot) \qu \text{on}\qu \bS^{n-1}.
\end{equation}
In addition, the total measure
\begin{equation}
S_p(K,\bS^{n-1})=\int_{\bS^{n-1}} \,\mathrm{d}S_p(K,u)=\int_{\bS^{n-1}} h_K^{1-p}(u) \,\mathrm{d}S(K, u)
\end{equation}
is called the \textit{$L_p$ surface area of $K$}.
Note that the $L_p$ surface area measure can be defined for all $p \in \mathbb{R}$ through the relation \eqref{eq:Sp-density}. 
The associated Minkowski type problem is then defined for all $p\in \R$ and is called the $L_p$ Minkowski problem. Precisely, it asks if a given measure $\mu$ on $\mathbb{S}^{n-1}$ arises as the $L_p$ surface area measure $S_p(K,\cdot)$ of a convex body $K \in \mathcal{K}^{n}_{o}$. This problem includes the classical Minkowski problem ($p=1$), the logarithmic Minkowski problem  ($p=0$), and the centro-affine Minkowski problem ($p=-n$) as important cases. 

On the other hand, an important variant of the logarithmic Minkowski problem with a parameter $q$ is recently proposed in \cite{HLYZ16} which incorporates the Aleksandrov problem ($q=0$) as well as the logarithmic Minkowski problem itself ($q=n$). Therefore, in view of the logarithmic Minkowski problem, there are two parameters $p$ and $q$ which produce two corresponding Brunn--Minkowski type theory, namely, the $L_p$ Brunn--Minkowski theory and the dual Brunn--Minkowski theory, respectively. Furthermore, in the very recent work \cite{LYZ18}, Lutwak, Yang, and Zhang unify the $L_p$ Minkowski problem and the dual Minkowski problem by considering two parameters $p$ and $q$ at the same time. In particular, they introduce the $L_p$ dual curvature measure $\widetilde C_{p,q}(K,\cdot)$ on $\bS^n$. The associated Minkowski problem is called the $L_p$ dual Minkowski problem, and this covers all Minkowski type problems mentioned so far. Since then, the $L_p$ dual Minkowski problem has been intensively studied by many authors, e.g., \cite{BF19,CHZ19,CCL21,CL21,HZ18_AM,LLL22}.

To describe the $L_p$ dual curvature measure $\widetilde{C}_{p,q}(K,\cdot)$, we recall the {\it $q$-th dual volume of $K$}:
\begin{align}
\widetilde {\mathrm{Vol}}_q(K)= \fr{1}{n}\int_{\bS^{n-1}}\rho^q(\xi) \,\mathrm{d}\xi,
\end{align}
where $\rho$ is the {\it radial function of $K$} defined by $\rho(\xi)=\max\{\lambda:\lambda \xi\in K\}$. For $q\not=0$, the {\it $q$-th dual curvature measure $\widetilde{C}_q(K,\cdot)$ of a convex body $K$} is determined by the following variational formula
\begin{equation}\label{eq:C_Kq}
\left. \frac{\mathrm{d}\widetilde {\mathrm{Vol}}_q(K+tL)}{\mathrm{d}t} \right|_{t = 0^{+}} = q\int_{\mathbb{S}^{n-1}} h_{L}h_{K}^{-1} \,\mathrm{d}\widetilde C_q(K,\cdot) \quad \text{for any }L\in \mathcal{K}_o^n.
\end{equation}
As in \eqref{eq:S_Kp}, we involve the $p$-linear combination into \eqref{eq:C_Kq} to produce the $L_p$ dual curvature measure $\widetilde C_{p,q}(K,\cdot)$, and it turns out that $\widetilde C_{p,q}(K,\cdot)$ is related to $\widetilde C_{q}(K,\cdot)$ by
\begin{align}
\widetilde C_{p,q}(K,\cdot)=h_K^{-p}\widetilde C_{q}(K,\cdot).
\end{align}

In the smooth category, the $L_p$ dual Minkowski problem becomes solving the Monge--Amp\`ere type equation
\begin{equation} \label{eq:MP}
h_K^{1-p}(h_K^2+|\D h_K|^2)^{\fr{q-n}{2}}\det(\nabla^2 h_K + h_K I) =  f \qu \text{on} \qu \mathbb{S}^{n-1},
\end{equation}
where $\nabla$ is the covariant derivative with respect to an orthonormal frame on $\mathbb{S}^{n-1}$ and $I$ is the identity matrix.
In this paper, we consider uniform diameter estimates of solutions to \eqref{eq:MP}.
This type of estimate is important in view of compactness results as in other geometric problems such as the Yamabe problem \cite{KMS09}. 

The diameter estimate in the case $p>q$ follows from a maximum principle argument. Indeed, if $h_K$ attains its maximum at a point $x_0\in\bS^n$, then the equation \eqref{eq:MP} implies that
\begin{align}\label{ineq:hmax}
\max_{\bS^{n-1}}h_K=h_K(x_0) \le  (\min_{\bS^{n-1}}f)^{\fr{1}{q-p}}.
\end{align}

However, the situation for $p\le q$ is complicated even in the $L_p$ Minkowski problem, i.e., $q=n$.
In fact, there are no diameter estimates for \eqref{eq:MP} when $-n<p<0$ and $q=n$ as shown in \cite{JLW15}, where Jian, Lu, and Wang construct a sequence of bounded positive functions $f_k$ such that diameters of the corresponding solutions are unbounded.
Note that the affine critical case $p=-n$ and $q=n$ also does not have diameter estimates since it is invariant under affine transformations. Much less is known about the super-critical case $p<-n$ and $q=n$ in which even the existence problem in the smooth category is resolved very recently \cite{GLW22}.

On the other hand, when $p>1$ and $q=n$, Chou and Wang \cite{CW06} showed the diameter estimates for solutions to \eqref{eq:MP}. 
In that paper, the case $1<p\le n$ is obtained by an approximation argument, and the case $p>n$ uses a simple maximum argument as in \eqref{ineq:hmax}. 
For the logarithmic case $p=0$ and $q=n$, there are recent results on the lower dimensional case. Namely, when $n=2$ or $3$, diameter estimates are provided by Chen and Li \cite{CL18} for $n=2$; by Chen, Feng, and Liu \cite{CFL22} for $n=3$. In particular, the paper \cite{CL18} concerns the diameter estimates in a more general framework of the dual Minkowski problem with $q>0=p$. Finally, the case $0<p<1$ and $q=2$ in planar $L_p$ Minkowski problem is treated in \cite{Du21}.

The question of diameter estimate remains open widely for general $(p,q)$ with $p<q$.
In this paper, we prove the diameter estimate for solutions to the planar $L_p$ dual Minkowski problem for $p \in (0,1)$ and $q\ge 2$.

\begin{theorem} \label{thm:est}
Let $p \in (0,1)$, $q \geq 2$, $\Lambda \geq 1$, $f \in L^{\infty}(\mathbb{S}^{1})$, and suppose that $1/\Lambda \leq f \leq \Lambda$. If a convex body $K \in \mathcal{K}^{1}_{o}$ solves \eqref{eq:MP}, then $\|h_K\|_{L^{\infty}(\mathbb{S}^{1})} \leq C$ for some $C = C(p, q, \Lambda) > 0$.
\end{theorem}

We stress that a crucial ingredient in \Cref{thm:est} is establishing estimates on the total measure of $\widetilde C_{p,q}(K,\cdot)$ in terms of its volume and surface area (\Cref{lem:Lp-ellipse}), which also provides a new proof in the $L_p$ case with $0<p<1$. Once we have the estimates, we will argue by contradiction assuming that there exists a sequence of convex bodies such that eccentricity of their John ellipse converges to one, i.e., the ratio of the major axis over the minor axis goes to infinity. 
As indicated in the equation \eqref{eq:MP}, the location of the origin is important since the equation becomes degenerate at points where $h_K$ is zero. Hence we divide the problem into two cases: (1) the origin lies near a tip; (2) the origin lies far from the tip. See details in \Cref{sec:upper-bound}.

It would be interesting to prove the diameter estimate for higher dimensional cases $n \geq 3$. We also note that positive lower bounds for solutions to \eqref{eq:MP} are not available in general even when $f$ is positive and smooth since there exists a counterexample, see the discussion for positivity of solutions below.

\vspace{0.5cm}

The diameter estimate can be applied to obtain the uniqueness and positivity of solutions to the $L_p$ Minkowski problem ($q=2$) when the density of the measure is sufficiently close to a constant in $C^\alpha$. 

\begin{theorem} \label{thm:MP}
Let $p \in (0,1)$, $q=2$, and $f \in C^{\alpha}(\mathbb{S}^{1})$. Then, there exists a small constant $\varepsilon_0 \in (0,1)$ that only depends on $p$, such that if $\|f-1\|_{C^{\alpha}(\mathbb{S}^{1})} \leq \varepsilon_0$, the equation \eqref{eq:MP} has a unique solution. Moreover, the solution is positive and of $C^{2,\alpha}(\mathbb{S}^{1})$.
\end{theorem}

There are relatively a few known results concerning the uniqueness of solutions for the case $0\le p<1$. This is because even in $n=2$ and $p=0$, the uniqueness does not hold for some positive, smooth function $f$ as shown in \cite{Yagisita06}. See also \cite{CLZ17} for non-uniqueness result when the given measure is discrete and $0<p<1$. 
Hence most of literature assume symmetry on $f$ \cite{CHLL20} or even strongly $f\equiv 1$ \cite{Gage93,Andrews99,BCD17}. For a non-symmetric $f$, the uniqueness of solutions to \eqref{eq:MP} is proved in \cite{CFL22} when $p=0$, $n=3$, and $f$ is close enough to a constant function as in \Cref{thm:MP}; and in \cite{BLYZ12} when $0<p<1$, $n=2$, and $f$ is origin symmetric.

It is worth noting that when $q>p+4$ and $p>0$, the solution to \eqref{eq:MP} is not unique even if $f\equiv1$, see \cite{CCL21}; when $p>q$, the solution to \eqref{eq:MP} is unique for any $f\in C^{\alpha}$, see \cite{HZ18_AM}. To the best of our knowledge, in the case $p<q\le p+4$ and $0<p<1$, the uniqueness for the equation \eqref{eq:MP} is not known even when $f\equiv1$.

Let us review some results related to the positivity of solutions for various $p$ with $q=2$. In \cite{CW06} the authors show that when $p\ge n$ or $-n<p\le -n+2$, solutions of \eqref{eq:MP} are positive under the assumption $0<c\le f\le C$ for some constants $c$ and $C$. On the other hand, the remaining case in the sub-critical range, $-n+2<p<n$, allows solutions that have zero at some points. See the examples in \cite[Example 1.6]{BT17} and \cite[Example 4.2]{BBC20} for $-n+2<p<1$ and in \cite[Example 4.1]{HLYZ05} for $1<p<n$. Especially in the planar case $n=2$, the solutions are not positive in general if $0<p<2$. Therefore, the condition that the function $f$ is close enough to a constant function in $C^\alpha$ is necessary for \Cref{thm:MP}. 

The paper is organized as follows. 
In \Cref{sec:surface-area} we estimate the total measure of $\widetilde C_{p,q}$ of convex bodies in terms of an interpolation between the area and length of their John ellipses. This estimate is one of the main ingredient in the proof of \Cref{thm:est}. 
We then prove our first main result, \Cref{thm:est}, in \Cref{sec:upper-bound}, and finally the proof of \Cref{thm:MP} is given in \Cref{sec:solvability}.


\section{\texorpdfstring{$L_p$}{Lp} dual curvature measure of convex bodies} \label{sec:surface-area}


In this section, we provide upper and lower bounds for the total measures of the $L_p$ dual curvature measures of convex bodies in terms of principal radii of their John ellipses. Recall from \cite{HLYZ16,LYZ18} that
\begin{equation}
\widetilde{C}_{p, q}(K, \mathbb{S}^{1}) = \int_{\mathbb{S}^{1}} \rho^{q-2}(\mathscr{A}^{\ast}(u)) h_K^{1-p}(u) \,\mathrm{d}S(K,u),
\end{equation}
where $\mathscr{A}^{\ast}(u) = \lbrace \xi \in \mathbb{S}^{1}: N(\rho(\xi)\xi) = u \rbrace$ is the reverse radial Gauss mapping and $N: \partial K \to \mathbb{S}^{1}$ is the Gauss map. For notational convenience, we write $S_K = S(K, \cdot)$ in the sequel.

\begin{lemma} \label{lem:Lp-ellipse}
Let $K \in \mathcal{K}^{1}_{o}$ and let $E$ be the John ellipse such that $E \subset K \subset 2E$ with principal radii $r_1$ and $r_2$. If $p \in [0, 1]$ and $q \geq 2$, then
\begin{equation} \label{eq:Cpq}
c_1(r_1r_2)^{1-p} \left( r_1^2+r_2^2 \right)^{\frac{p+q-2}{2}} \leq \widetilde{C}_{p, q}(K, \mathbb{S}^{1}) \leq c_2 (r_1r_2)^{1-p} \left( r_1^2+r_2^2 \right)^{\frac{p+q-2}{2}}
\end{equation}
for some $c_1, c_2 > 0$ which only depend on $p$ and $q$.
\end{lemma}

\begin{proof}
Let us first prove the upper bound of $\widetilde{C}_{p, q}(K, \mathbb{S}^{1})$. Since
\begin{equation}
\rho(\mathscr{A}^{\ast}(u)) \leq 4\left( r_1^2 + r_2^2 \right)^{1/2}
\end{equation}
for all $u \in \mathbb{S}^{1}$, we obtain
\begin{equation} \label{eq:C-upper1}
\widetilde{C}_{p, q}(K, \mathbb{S}^{1}) \leq 4^{q-2} \left( r_1^2 + r_2^2 \right)^{\frac{q-2}{2}} S_p(K,\mathbb{S}^{1}).
\end{equation}
By using the H\"older inequality, we have
\begin{equation} \label{eq:C-upper2}
\begin{split}
S_p(K,\mathbb{S}^{1})
&\leq 2^{1-p}\mathrm{Area}(K)^{1-p} \mathcal{H}^{1}(\partial K)^p \\
&\leq 2^{1-p}\mathrm{Area}(2E)^{1-p} \mathcal{H}^{1}(\partial (2E))^p.
\end{split}
\end{equation}
In the second inequality we used that the area of a set and the length of the boundary of a convex body are monotone with respect to the set inclusion. Combining \eqref{eq:C-upper1}--\eqref{eq:C-upper2} and using
\begin{equation}
\mathrm{Area}(2E)= 4\pi r_1r_2 \quad\text{and}\quad \mathcal{H}^{1}(\partial (2E)) \leq 2\sqrt{2}\pi \left( r_1^2 + r_2^2 \right)^{1/2},
\end{equation}
we deduce the upper bound of $\widetilde{C}_{p, q}(K, \mathbb{S}^{1})$ in \eqref{eq:Cpq} with $c_2 = 2^{2q-1-3p/2}\pi$.

Let us next prove the lower bound of $\widetilde{C}_{p, q}(K, \mathbb{S}^{1})$. By rotating the coordinates if necessary, we may assume that $E$ is given by
\begin{equation}
E = \left\lbrace  (x_1, x_2) \in \mathbb{R}^2: \frac{(x_1-a_1)^2}{r_1^2}+\frac{(x_2-a_2)^2}{r_2^2} \leq 1 \right\rbrace
\end{equation}
with $a_1,a_2\ge0$. 
We consider two lines $L_+$ and $L_-$ that pass through the point $P = (a_1+2r_1, a_2+2r_2)$ and that are tangent to $\partial E$ at $Q_+$ and $Q_-$ (see \Cref{fig1}).
Then, we may write them as $x_2=l_+(x_1)$ and $x_2=l_-(x_1)$ with $l_-(a_1)<l_+(a_1)$, where $l_+$ and $l_-$ denote the linear functions whose graphs are $L_+$ and $L_-$, respectively.
We define a set
\begin{equation}
F_1 = \left\lbrace (x_1,x_2) \in \partial K: l_1(x_1) \leq x_2 \leq l_2(x_1)  \right\rbrace
\end{equation}
and its image $G_1 = N(F_1)$ under the Gauss map $N: \partial K \to \mathbb{S}^{1}$. 
We first want to obtain a lower bound of $h_K$ on $G_1$. 

\begin{figure}[h]
   \includegraphics[width=\linewidth]{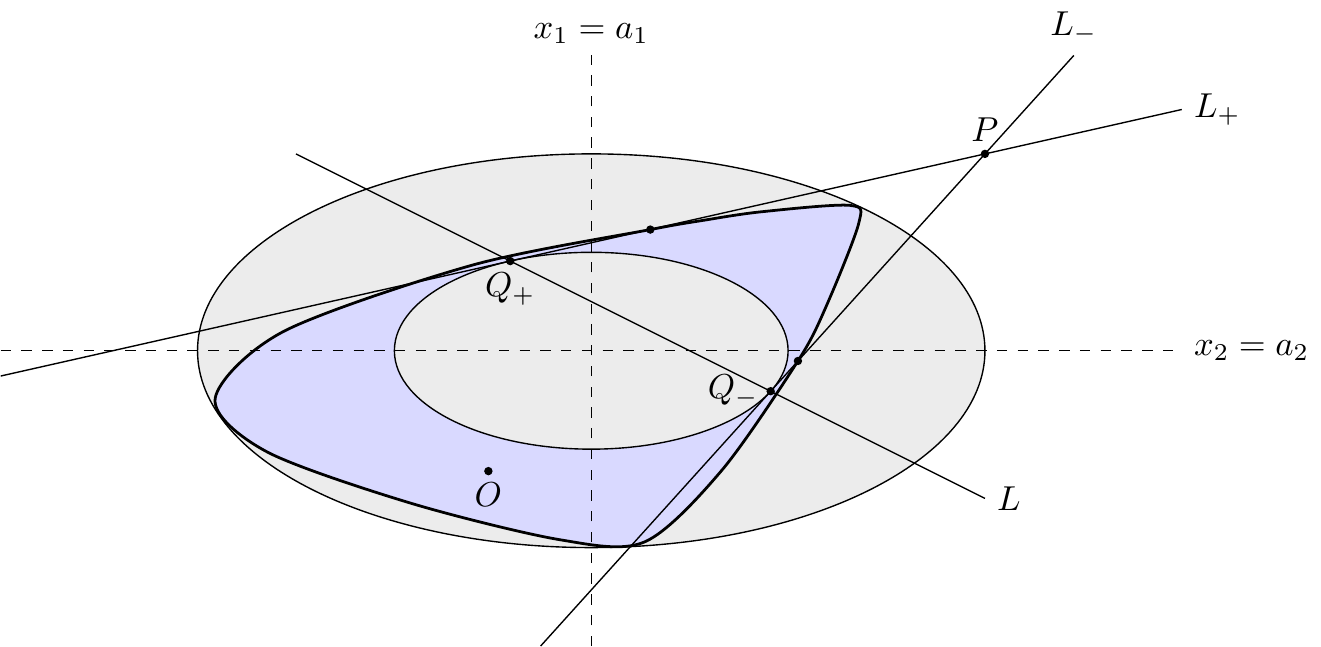} 
   \caption{}
   \label{fig1}
\end{figure}

To this end, we first observe that the lines $L_+$ and $L_-$ are given by
\begin{equation}
L_\pm: \frac{1\mp\sqrt{7}}{4r_1} (x_1-a_1) + \frac{1\pm\sqrt{7}}{4r_2} (x_2-a_2) = 1
\end{equation}
and that the contact points $Q_\pm$ between $\partial E$ and $L_\pm$ are written as 
\begin{align}
Q_\pm=(a_1,a_2)+\left(\tfrac{1\mp\sqrt{7}}{4}r_1,\tfrac{1\pm\sqrt{7}}{4}r_2\right).
\end{align}
Moreover, if we denote by $L$ the line passing through the points $Q_+$ and $Q_-$, then the equation for $L$ is simply
\begin{align}
\fr{x_1- a_1}{r_1}+\fr{x_2-a_2}{r_2}=\fr{1}{2}.
\end{align}

For $\nu \in G_1$, at least one of $\nu \cdot e_1 $ and $\nu \cdot e_2$ is nonnegative. 
If $\nu \cdot e_1 \geq 0$ and $\nu \cdot e_2 \geq 0$, then
\begin{equation}
h_K(\nu) \geq \mathrm{dist}((a_1, a_2), L) =\tfrac{1}{2}r_1r_2\left( r_1^2+r_2^2 \right)^{-1/2}.
\end{equation}
If $\nu \cdot e_1 \leq 0$ and $\nu \cdot e_2 \geq 0$, then we have
\begin{align}
h_K(\nu) &\geq \mathrm{dist}((a_1-2r_1, a_2), L_+)
\\
&=
\frac{|\frac{\sqrt{7}-1}{2}-1|}{\sqrt{\left( \frac{1-\sqrt{7}}{4r_1} \right)^2 + \left( \frac{1+\sqrt{7}}{4r_2} \right)^2}} \geq C r_1r_2 \left( r_1^2+r_2^2 \right)^{-1/2}
\end{align}
for some small $C > 0$. Since the same argument holds when $\nu \cdot e_1 \geq 0$ and $\nu \cdot e_2 \leq 0$, we have
\begin{equation} \label{eq:hK}
h_K(\nu) \geq C r_1r_2 \left( r_1^2+r_2^2 \right)^{-1/2}.
\end{equation}
in any cases. Therefore, it follows from \eqref{eq:hK} that
\begin{equation} \label{eq:Cpq-lower}
\begin{split}
\widetilde{C}_{p, q}(K, \mathbb{S}^{1})
&\geq \int_{G_1} h_K^{1-p} \rho^{q-2} \,\mathrm{d}S_K \\
&\geq C \left( r_1r_2 \left( r_1^2+r_2^2 \right)^{-1/2} \right)^{1-p} \int_{G_1} \rho^{q-2} \,\mathrm{d}S_K
\end{split}
\end{equation}
for some $C = C(p) > 0$.

It remains to obtain a lower bound of the integral on the right-hand side of \eqref{eq:Cpq-lower}. Let $P_+ = (a_1+\frac{6}{5}r_1, a_2+\frac{8}{5}r_2) \in \partial (2E)$ be one of the intersections of $\partial (2E)$ and the line passing through $P = (a_1+2r_1, a_2+2r_2)$ and $(a_1-2r_1, a_2)$. Similarly, let $P_- = (a_1+\frac{8}{5}r_1, a_2+\frac{6}{5}r_2) \in \partial (2E)$ be one of the intersections of $\partial (2E)$ and the line passing through $P$ and $(a_1, a_2-2r_2)$. Then, the lines passing through $P_\pm$ and $(a_1, a_2)$ are given by $x_2-a_2 = \frac{4r_2}{3r_1}(x_1-a_1)$ and $x_2-a_2 = \frac{3r_2}{4r_1}(x_1-a_1)$, respectively. See \Cref{fig2}.

\begin{figure}[h]
   \includegraphics[width=\linewidth]{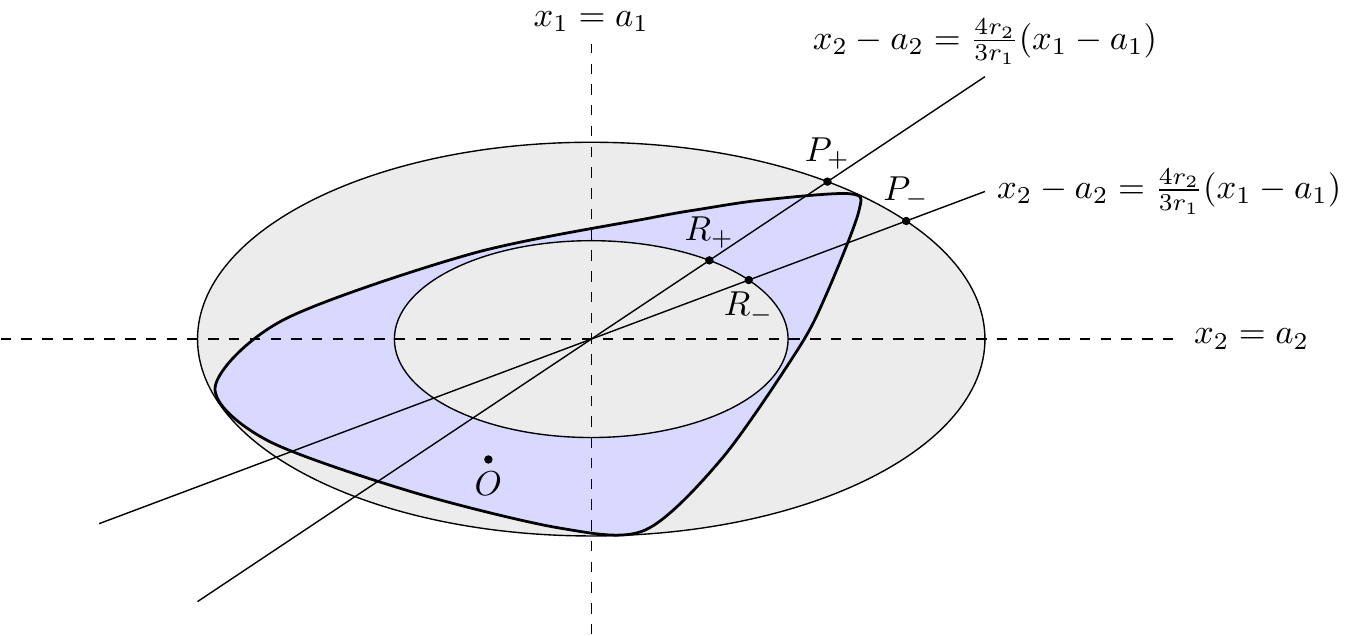} 
   \caption{}
   \label{fig2}
\end{figure}

We define a set
\begin{equation}
F_2 = \left\lbrace (x_1, x_2) \in \partial K: \tfrac{3r_2}{4r_1}(x_1-a_1) \leq x_2-a_2 \leq \tfrac{4r_2}{3r_1}(x_1-a_1) \right\rbrace
\end{equation}
and consider its image $G_2 = N(F_2)$ under the Gauss map $N$. Then, we have $F_2 \subset F_1$ and $G_2 \subset G_1$. Moreover, we obtain
\begin{equation} \label{eq:r-lower}
r \geq \min \lbrace \mathrm{dist}((a_1, a_2), R_+), \mathrm{dist}((a_1, a_2), R_-) \rbrace \geq \tfrac{3}{5} \left( r_1^2 + r_2^2 \right)^{1/2}
\end{equation}
on $G_2$, where $R_+ = (a_1+\frac{3}{5}r_1, a_2 + \frac{4}{5}r_2) \in \partial E \cap \lbrace x_2-a_2=\frac{4r_2}{3r_1}(x_1-a_1) \rbrace$ and $R_- = (a_1+\frac{4}{5}r_1, a_2 + \frac{3}{5}r_2) \in \partial E \cap \lbrace x_2-a_2 = \frac{3r_2}{4r_1} (x_1-a_1) \rbrace$. Furthermore, we have
\begin{equation} \label{eq:length-lower}
S_K(G_2) \geq \mathcal{H}^{1}(F_2) \geq \mathrm{dist}(R_+, R_-) = \tfrac{1}{5}\left( r_1^2+r_2^2 \right)^{1/2}.
\end{equation}
Therefore, \eqref{eq:Cpq-lower}, \eqref{eq:r-lower}, and \eqref{eq:length-lower} yield the lower bounds of $\widetilde{C}_{p, q}(K, \mathbb{S}^{1})$ in \eqref{eq:Cpq} for some $c_1 = c_1(p, q) > 0$.
\end{proof}


\section{Diameter estimate} \label{sec:upper-bound}


This section is devoted to the proof of \Cref{thm:est}. We basically follow the ideas given in \cite{CL18} and \cite{CFL22}, but our proof significantly uses \Cref{lem:Lp-ellipse}.

\begin{proof}[Proof of \Cref{thm:est}]
Let $E$ be the John ellipse such that
\begin{equation} \label{eq:ellipse}
E \subset K \subset 2E,
\end{equation}
where $E$ is centered at the center of mass of $K$. By rotating the coordinates if necessary, we may assume that $E$ is given by
\begin{equation}
E = \left\lbrace (x_1,x_2) \in \mathbb{R}^2: \frac{(x_1-a_1)^2}{r_1^2}+\frac{(x_2-a_2)^2}{r_2^2} \leq 1 \right\rbrace
\end{equation}
with $r_1 \geq r_2$. Once we prove that there exists a constant $C_0 > 0$ such that $r_2 \geq C_0 r_1$, the desired result follows. Indeed, if $r_2 \geq C_0 r_1$, then it follows from \Cref{lem:Lp-ellipse} that
\begin{equation}
c_1 C_0^{1-p} r_1^{q-p} \leq c_1(r_1r_2)^{1-p} \left( r_1^2+r_2^2 \right)^{\frac{p+q-2}{2}} \leq \widetilde{C}_{p, q}(K, \mathbb{S}^{1}) = \int_{\mathbb{S}^{1}} f \,\mathrm{d}\mathcal{H}^{1} \leq 2\pi \Lambda,
\end{equation}
which shows that
\begin{equation}
\|h_K\|_{L^{\infty}(\mathbb{S}^{1})} \leq \mathrm{diam}(K) \leq 4r_1 \leq 4\left( \frac{2\pi \Lambda }{c_1 C_0^{1-p}} \right)^{\frac{1}{q-p}}.
\end{equation}
Therefore, we will prove that $r_2 \geq C_0 r_1$ in the rest of the proof.

Assume to the contrary that, for arbitrarily large $M > 0$, there exists a convex body $K = K_M$ satisfying \eqref{eq:MP} for some $f_K$ with $1/\Lambda \leq f_K \leq \Lambda$ and having $r_1 > Mr_2$. Let us denote by $I = [a, b]$ the projection of $K$ onto the $x_1$-axis, and let $d= \mathrm{dist}(0, \partial I)$. There are two possibilities:

\begin{enumerate}[1.]
\item
There exists a constant $c_0 \in (0, 1/2)$ such that, for any $M > 0$, the convex body $K = K_M$ satisfies
\begin{equation}
d \geq c_0 r_1.
\end{equation}
\item
There exists a sequence $\lbrace M_i \rbrace_{i \in \mathbb{N}}$ of real numbers such that $M_i \to \infty$ as $i \to \infty$ and that the convex body $K_i = K_{M_i}$ satisfies
\begin{equation} \label{eq:case2}
d < \frac{r_1}{M_i}.
\end{equation}
\end{enumerate}

Let us first consider the case 1. We set
\begin{equation}
F = \lbrace (x_1,x_2) \in \partial K: \mathrm{dist}(x_1, \partial I) \geq \varepsilon_0 r_1 \rbrace,
\end{equation}
where $\varepsilon_0 = c_0/2$. Then, $F$ consists of two disjoint connected components $F = F_{+} \cup F_{-}$, where $F_{+}$ and $F_{-}$ are given by
\begin{equation}
F_{\pm} = \lbrace \gamma_{\pm}(s)=(s, \varphi_{\pm}(s)): s_0 \leq s \leq s_1 \rbrace
\end{equation}
for some concave function $\varphi_{+}$ and convex function $\varphi_{-}$. Here, $s_0 = a+ \varepsilon_0 r_1$ and $s_1= b-\varepsilon_0 r_1$. Let $N: \partial K \to \mathbb{S}^{1}$ be the Gauss map and define
\begin{equation}
G = N(F), \quad G_{+} = N(F_{+}), \quad\text{and}\quad G_{-} = N(F_{-}).
\end{equation}
On the one hand, we have
\begin{equation}
\int_{G} f_K \,\mathrm{d}\mathcal{H}^{1} \to 0
\end{equation}
as $M \to \infty$ since $G$ converges to $\lbrace e_2, -e_2 \rbrace$ as $M \to \infty$. On the other hand, we claim that there exists a constant $C > 0$ such that
\begin{equation} \label{eq:claim1}
\int_{G} \rho^{q-2} h_K^{1-p} \,\mathrm{d}S_K \geq C \int_{\mathbb{S}^{1}} \rho^{q-2} h_K^{1-p} \,\mathrm{d}S_K.
\end{equation}
Since
\begin{equation}
\int_{\mathbb{S}^{1}} \rho^{q-2} h_K^{1-p} \,\mathrm{d}S_K = \int_{\mathbb{S}^{1}} f_K \,\mathrm{d}\mathcal{H}^{1} \geq \fr{2\pi}{\Lambda},
\end{equation}
the claim \eqref{eq:claim1} will lead us to a contradiction as $M \to \infty$.

To prove \eqref{eq:claim1}, let us consider the quadrangle $\mathcal{S}$ with vertices $(s_0, \varphi_{\pm}(s_0))$ and $(s_1, \varphi_{\pm}(s_1))$. Let $\nu_{\pm}$ denote the outer unit normal vectors to $\mathcal{S}$ that are perpendicular to the lines $L_{\pm}$ joining $(s_0, \varphi_{\pm}(s_0))$ and $(s_1, \varphi_{\pm}(s_1))$, respectively. We may assume that $h_{\mathcal{S}}(\nu_{+}) \geq h_{\mathcal{S}}(\nu_{-})$. The square $\mathcal{S}$ may not contain the origin, but at least one of $h_{\mathcal{S}}(\nu_{+})$ and $h_{\mathcal{S}}(\nu_{-})$ is positive. Therefore, we have $h_{\mathcal{S}}(\nu_{+}) > 0$, see \Cref{fig3}.

\begin{figure}[h]
   \includegraphics[width=.7\linewidth]{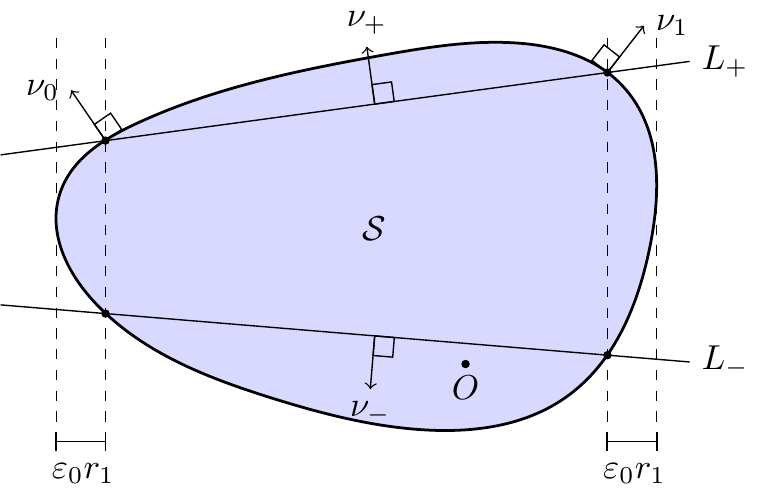}
   \caption{}
   \label{fig3}
\end{figure}

Since $K$ is convex, $G_+$ must be a connected arc of $\bS^1$. We denote by $\nu_0$ and $\nu_1$ the end points of the arc $G_+$ such that $\nu_0\cdot e_1\le\nu_1\cdot e_1$. Then it follows that $h_K(u) \geq h_{\mathcal{S}}(u) \geq \min \{h_{\mathcal{S}}(\nu_{+}), h_{\mathcal{S}}(\nu_0),h_{\mathcal{S}}(\nu_1)\}=:h_{\min}$ for all $u \in G_+$. Using this inequality, we obtain
\begin{equation} \label{eq:LHS1}
\int_{G} \rho^{q-2} h_K^{1-p} \,\mathrm{d}S_K \geq \int_{G_{+}} \rho^{q-2} h_K^{1-p} \,\mathrm{d}S_K \geq h^{1-p} \int_{G_+} \rho^{q-2} \,\mathrm{d}S_K.
\end{equation}
To estimate $h_{\min}$, we will show that
\begin{equation} \label{eq:LHS2}
h_{\min} \geq c_1 \min\lbrace \varphi_{+}(s_0)-\varphi_{-}(s_0), \varphi_{+}(s_1)-\varphi_{-}(s_1) \rbrace \geq c_2 r_2
\end{equation}
for some $c_1, c_2 > 0$, provided that $M>1$. For the second inequality in \eqref{eq:LHS2}, let $A$ be an intersection point of $K$ and $\lbrace x_1 = a \rbrace$. Let $\tilde L_{\pm}$ be the lines passing through $A$ and $(s_0, \varphi_{\pm}(s_0))$, and $A_{\pm}$ be the intersections of $\tilde L_{\pm}$ and $\lbrace x_1 = a_1 \rbrace$, respectively. Then we obtain
\begin{equation}
\varphi_{+}(s_0) - \varphi_{-}(s_0) = \varepsilon_0 r_1 \frac{|A_+A_-|}{a_1-a} \geq \varepsilon_0r_1 \frac{2r_2}{2r_1} = \varepsilon_0 r_2,
\end{equation}
and in a similar way, we also have $\varphi_+(s_1)-\varphi_-(s_1)\ge \varepsilon_0r_2$ which proves the second inequality in \eqref{eq:LHS2}.

\begin{figure}[h]
   \includegraphics[width=.65\linewidth]{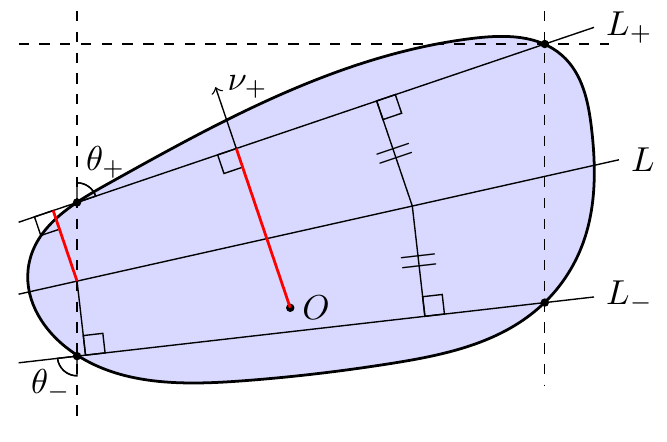} 
   \caption{}
   \label{fig4}
\end{figure}

For the first inequality in \eqref{eq:LHS2}, let us assume that $\varphi_{+}(s_0) - \varphi_{-}(s_0) \leq \varphi_{+}(s_1) - \varphi_{-}(s_1)$. Recall that $L_{\pm}$ are lines joining $(s_0, \varphi_{\pm}(s_0))$ and $(s_1, \varphi_{\pm}(s_1))$, respectively. If we consider a line
\begin{equation}
L = \lbrace (s, l(s)): s \in (s_0, s_1), \mathrm{dist}((s, l(s)), L_+) = \mathrm{dist}((s, l(s)), L_-) \rbrace,
\end{equation}
then the origin $O$ lies below the line $L$. 
Denoting the intersection point of $L$ and $x_1=s_0$ by $(s_0,m)$, it follows from the definition of $L$ that
\begin{equation}\label{ineq:hS}
h_{\mathcal{S}}(\nu_{+}) \geq \mathrm{dist}((s_0, m), L_+) = (\varphi_{+}(s_0) - m) \sin \theta_+=(m-\varphi_{-}(s_0) ) \sin \theta_-,
\end{equation}
where $\theta_\pm$ is the angle between $L_\pm$ and $x_2$-axis, see \Cref{fig4}. 

Since $\varphi_\pm(s_1)-\varphi_\pm(s_0) \leq 4r_2$ and $s_1-s_0 \geq 3r_1/2$, we have
\begin{equation}
\sin\theta_\pm \geq \tfrac{3}{2}r_1 \left[(\tfrac{3}{2}r_1)^2 + (4r_2)^2\right]^{-1/2}.
\end{equation}
Since $r_1 > Mr_2$, we have $\sin \theta_\pm > 2c_1$ for some $c_1$, provided that $M > 1$. 
Using this, we obtain
\begin{align}\label{ineq:dist}
\mathrm{dist}((s_0,m),L_+)
&= \tfrac{1}{2}(\varphi_+(s_0)-m)\sin\theta_+
+\tfrac{1}{2}(m-\varphi_-(s_0))\sin\theta_-
\\
&\ge c_1(\varphi_+(s_0)-\varphi_-(s_0))
\end{align}
which implies $h_{\mathcal{S}}(\nu_+)\ge c_1(\varphi_+(s_0)-\varphi_-(s_0))$.

On the other hand, if $\theta_0$ is the angle between $x_2$-axis and the tangent line at $(s_0,\varphi_+(s_0))$ to $\partial K$, then we have
\begin{align}\label{ineq:hSnu0}
h_{\mathcal{S}}(\nu_0)\ge (\varphi_+(s_0)-m)\sin \theta_0,
\end{align}
see \Cref{fig5}.
Using \eqref{ineq:dist}, we observe that $\varphi_+(s_0)-m\ge \mathrm{dist}((s_0,m),L_+)\ge c_1(\varphi_+(s_0)-\varphi_-(s_0))$. If we denote by $\phi$ the angle between $x_2$-axis and the line passing through the points $(s_0,\varphi_+(s_0))$ and a left tip of $\partial K$, then clearly $\sin \phi\le \sin \theta_0$. It then follows from $r_1>r_2$ and $\varphi_+(s_0)-\varphi_-(s_0)\le 4r_2$ that
\begin{align}
\sin\phi\ge \fr{\varepsilon_0 r_1 }{ \sqrt{[\varphi_+(s_0)-\varphi_-(s_0)]^2+(\varepsilon_0r_1)^2}}\ge \fr{\varepsilon_0}{\sqrt{\varepsilon_0^2+16}}.
\end{align}
Thus we conclude $h_{\mathcal{S}}(\nu_0)\ge c_1 (\varphi_+(s_0)-\varphi_-(s_0))$ for some smaller $c_1$. In a similar way, we obtain the same lower bound for $h_{\mathcal{S}}(\nu_1)$, and therefore \eqref{eq:LHS2} holds.

\begin{figure}[h]
   \includegraphics[width=.65\linewidth]{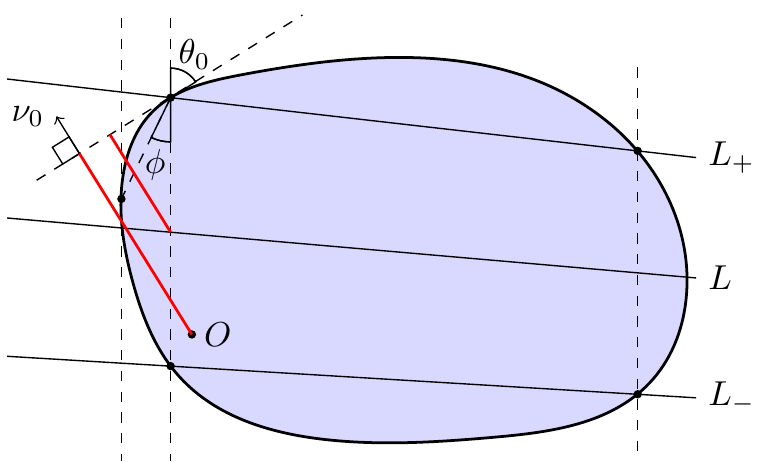} 
   \caption{}
   \label{fig5}
\end{figure}

It remains to estimate the integral on the right-hand side of \eqref{eq:LHS1}. To this end, we may assume that $|a| \leq |b|$ and define sets
\begin{equation}
F_{\ast} = \lbrace \gamma_+(s): s_{\ast} \leq s \leq s_1 \rbrace \subset F_+, \qu G_{\ast}=N(F_{\ast}),
\end{equation}
where $s_{\ast} = s_1-(s_1-s_0)/4$. In this set, we have a lower bound of $\rho$:
\begin{equation} \label{eq:LHS3}
\rho\geq \min\lbrace \mathrm{dist}(O, \gamma_+(s_{\ast})), \mathrm{dist}(O, \gamma_+(s_1)) \rbrace \geq (s_1-s_0)/4.
\end{equation}
Moreover, we know that $S_K(G_{\ast}) \geq (s_1-s_0)/4$. Note that it follows from \eqref{eq:ellipse} that $2r_1 \leq \mathcal{H}^{1}(I) = b-a \leq 4r_1$. Thus, we obtain
\begin{equation} \label{eq:LHS4}
s_1-s_0 = b-a-2\varepsilon_0 r_1 \geq (2-c_0)r_1 \geq \tfrac{3}{2} r_1.
\end{equation}
Thus, it follows from \eqref{eq:LHS1}, \eqref{eq:LHS2}, \eqref{eq:LHS3}, and \eqref{eq:LHS4} that
\begin{equation}
\int_{G} \rho^{q-2} h_K^{1-p} \,\mathrm{d}S_K \geq C r_1^{q-1} r_2^{1-p}.
\end{equation}
Hence, by using $r_1 \geq r_2$ and \Cref{lem:Lp-ellipse}, we obtain
\begin{equation}
\int_{G} \rho^{q-2} h_K^{1-p} \,\mathrm{d}S_K \geq C (r_1r_2)^{1-p} \left( r_1^2+r_2^2 \right)^{\frac{p+q-2}{2}} \geq C \int_{\mathbb{S}^{1}} \rho^{q-2} h_K^{1-p} \,\mathrm{d}S_K,
\end{equation}
which concludes \eqref{eq:claim1}.

Let us next consider the case 2. In this case, we may assume that $d = |a| < |b|$ so that $K_i \subset \lbrace (x_1,x_2) \in \mathbb{R}^2: x_1 \geq -d \rbrace$. We set
\begin{equation}
F_i = \left\lbrace (x_1,x_2) \in \partial K_i: x_1 \geq \frac{r_1}{M_i} \right\rbrace
\end{equation}
and $G_i = N(F_i)$. 
Then we can prove that
\begin{equation} \label{eq:omega-i}
G_i \subset \mathbb{S}_{0} := \lbrace u \in \mathbb{S}^{1}: u \cdot e_1 \geq -2/\sqrt{5} \rbrace
\end{equation}
for sufficiently large $i$. 
Indeed, if this is not the case, then one of the intersection points of $\partial K_i$ and the line $\{x_1=r_1/M_i\}$ has a unit normal vector that is not in $\mathbb{S}_0$ for infinitely many $i$. 
Let $P_{i}$ be the intersection point, with the larger $x_2$-coordinates, of $\partial K_i$ and the line $\lbrace x_1 = r_1/M_i \rbrace$, and let $v_i= (v_{i, 1}, v_{i, 2}) \in \mathbb{S}^{1}$ be the unit normal vector of $\partial K_i$ at $P_i$. We may assume that $v_i\not \in \mathbb{S}_0$ for infinitely many $i$. Then for such $i$, it can be verified that $v_{i,1}+2v_{i,2}<0$. On the other hand, since the convex body $K_i$ is contained in the half space
\begin{equation}
H_i :=\left\lbrace (x_1,x_2) \in \mathbb{R}^2: \left( x_1- \frac{r_1}{M_i}, x_2 - 2r_2 \right) \cdot v_i \leq 0 \right\rbrace
\end{equation}
and the origin $O$ lies in $K$, it follows from $r_1>M_ir_2$ and $v_{i,1}<0$ that 
\begin{equation}
0\ge-\frac{r_1}{M_i} v_{i, 1} - 2r_2 v_{i, 2} > -(v_{i, 1} + 2v_{i, 2}) r_2. 
\end{equation}
However, the last term in this inequality is positive, which is a contradiction. Therefore, \eqref{eq:omega-i} is true for sufficiently large $i$, and hence,
\begin{equation} \label{eq:lower}
\frac{1}{\Lambda} |\mathbb{S}^{1} \setminus \mathbb{S}_{0}| \leq \int_{\mathbb{S}^{1} \setminus \mathbb{S}_{0}} f_{K_i} \,\mathrm{d}\mathcal{H}^{1} \leq \int_{\mathbb{S}^{1} \setminus G_i} f_{K_i} \,\mathrm{d}\mathcal{H}^{1} = \widetilde{C}_{p, q}(K_i, \mathbb{S}^{1} \setminus G_i).
\end{equation}
We claim that
\begin{equation} \label{eq:claim2}
\widetilde{C}_{p, q}(K_i, \mathbb{S}^{1} \setminus G_i) \leq \frac{C}{M_i^{1-p}} (r_1r_2)^{1-p} \left( r_1^2+r_2^2 \right)^{\frac{p+q-2}{2}}.
\end{equation}
Once we prove \eqref{eq:claim2}, we will have from \eqref{eq:lower}, \eqref{eq:claim2}, and \Cref{lem:Lp-ellipse} that
\begin{equation}
CM_i^{1-p} \leq \widetilde{C}_{p, q}(K_i, \mathbb{S}^{1}) = \int_{\mathbb{S}^{1}} f_{K_i} \,\mathrm{d}\mathcal{H}^{1} \leq 2\Lambda \pi,
\end{equation}
which yields a contradiction by sending $i \to \infty$.

Let us now prove the claim \eqref{eq:claim2}. First of all, we have $r \leq 4(r_1^2+r_2^2)^{1/2}$ on $\mathbb{S}^{1}$, and hence,
\begin{equation}
\widetilde{C}_{p, q}(K_i, \mathbb{S}^{1} \setminus G_i) \leq 4^{q-2}\left( r_1^2+r_2^2 \right)^{\frac{q-2}{2}} \int_{\mathbb{S}^{1} \setminus G_i} h_{K_i}^{1-p} \,\mathrm{d}S_{K_i}.
\end{equation}
We define
\begin{equation}
\tilde{K}_i = \left\lbrace x \in K_i: x \cdot e_1 \leq \frac{r_1}{M_i} \right\rbrace
\end{equation}
and
\begin{equation}
Q_i = \left\lbrace x \in \mathbb{R}^{2}: -d \leq x_1 \leq \frac{r_1}{M_i}, -2r_2 \leq x_2-a_2 \leq 2r_2 \right\rbrace.
\end{equation}
Since $\tilde{K}_i \subset Q_i$, we have $\mathrm{Area}(\tilde{K}_i) \leq \mathrm{Area}(Q_i)$. Moreover, we also have $\mathcal{H}^{1}(\partial \tilde{K}_i) \leq \mathcal{H}^{1}(\partial Q_i)$ since $\tilde{K}_i$ is convex. By the H\"older inequality, we obtain
\begin{equation}
\begin{split}
\int_{\mathbb{S}^{1} \setminus G_i} h_{K_i}^{1-p} \,\mathrm{d}S_{K_i}
&\leq \left( \int_{\mathbb{S}^{1}\setminus G_i} h_{K_i} \,\mathrm{d}S_{K_i} \right)^{1-p} \left( \int_{\mathbb{S}^{1}\setminus G_i} \,\mathrm{d}S_{K_i} \right)^p \\
&\leq 2^{1-p}\mathrm{Area}(\tilde{K}_i)^{1-p} \mathcal{H}^{1}(\partial \tilde{K}_i)^p \\
&\leq 2^{1-p}\mathrm{Area}(Q_i)^{1-p} \mathcal{H}^{1}(\partial Q_i)^p \\
&= 2^{1-p}\left( 4r_2 \left( \frac{r_1}{M_i}+d \right) \right)^{1-p} \left( \frac{2r_1}{M_i}+2d+ 8r_2 \right)^p.
\end{split}
\end{equation}
Thus, the claim \eqref{eq:claim2} follows from \eqref{eq:case2} and $(r_1+r_2)^p \leq 2^{p/2} (r_1^2+r_2^2)^{p/2}$, provided that $M_i > 1$.
\end{proof}


\section{Uniqueness and positivity of solution} \label{sec:solvability}


In this section, we prove the uniqueness and positivity of a solution to the $L_p$ Minkowski problem.

\begin{lemma} \label{lem:uniqueness}
Let $p \in (0, 1)$, $q=2$, and $f \in C^{\alpha}(\mathbb{S}^{1})$. Then, there exists a constant $\varepsilon_0 \in (0, 1/2)$ such that if $\|f-1\|_{C^{\alpha}(\mathbb{S}^{1})} \leq \varepsilon_0$ and if $K_j$, $j=1, 2$, are solutions to \eqref{eq:MP} with $\|h_{K_j}-1\|_{L^{\infty}(\mathbb{S}^{1})} \leq \varepsilon_0$, then $K_1=K_2$.
\end{lemma}

\begin{proof}
Suppose to the contrary that there is no such a constant $\varepsilon_0$. Then, there exist sequences $\{K_i\}_{i\in \mathbb{N}}$ and $\{L_i\}_{i\in \mathbb{N}}$ of convex bodies such that $K_i\not = L_i$, 
\begin{align}
\|h_{K_i}-1\|_{L^{\infty}(\mathbb{S}^{1})} \le \fr{1}{i+1}, \qu \|h_{L_i}-1\|_{L^{\infty}(\mathbb{S}^{1})} \le \fr{1}{i+1}, 
\end{align}
and both $K_i$ and $L_i$ are solutions to \eqref{eq:MP} with $f_i$ satisfying $\|f_i-1\|_{C^\alpha(\mathbb{S}^{1})} \le \fr{1}{i+1}$. By the result of Caffarelli \cite{Caf90}, we have for large $i$ that $h_{K_i}\in W^{2,q}$ for all $q>0$ and hence $h_{K_i}\in C^{1,\alpha'}$ for all $0<\alpha'<1$.
Moreover, we have $\|h_{K_i}\|_{C^{1,\alpha'}(\mathbb{S}^{1})} \le C$ for some constant $C>0$ independent of $i$. By compactness, we find a subsequence (still denoted by $\{ K_i\}$) such that $h_{K_i} \rightarrow 1$ in $C^{1,\alpha''}(\mathbb{S}^{1})$ for all $0<\alpha''<\alpha'$ as $i\rightarrow \infty$, and in particular $h_{K_i}\rightarrow 1$ in $C^{0,1}(\mathbb{S}^{1})$. 
Furthermore, $h_{K_i}^{p-1} \rightarrow 1$ in $C^{0,1}(\mathbb{S}^{1})$ since $h_{K_i}\ge 1/2$ and so for any $x,y \in \bS^1$
\begin{align}
\fr{1}{1-p}\left| h_{K_i}^{p-1}(x) - h_{K_i}^{p-1}(y) \right|
\le \left|\int_{h_{K_i}(y)}^{h_{K_i}(x)} t^{p-2} \,\mathrm{d}t \right|
\le 2^{2-p}|h_{K_i}(x)-h_{K_i}(y)|.
\end{align}
Thus, it follows from $f_i\rightarrow 1$ in $C^{\alpha}(\mathbb{S}^{1})$ that
\begin{align}
h_{K_i}^{p-1}f_i\rightarrow 1 \qu \text{in } C^\alpha(\mathbb{S}^{1}) \qu \text{as }i \rightarrow \infty.
\end{align}

Note that $h_{K_i}-1$ satisfies a uniformly elliptic linear equation whose coefficients are independent in $i$. Indeed, it satisfies
\begin{align}
\Delta_{\mathbb{S}^{1}} (h_{K_i}-1) + (h_{K_i}-1) &=h_{K_i}^{p-1}f_i-1 \quad \text{on } \mathbb{S}^{1}.
\end{align}
Applying the Schauder estimate to $h_{K_i}-1$, we have
\begin{align}
\|h_{K_i}-1\|_{C^{2,\alpha}(\mathbb{S}^{1})} \le C\left( \|h_{K_i}-1\|_{L^{\infty}(\mathbb{S}^1)} + \left\| h_{K_i}^{p-1}f_i-1 \right\|_{C^\alpha(\mathbb{S}^{1})} \right),
\end{align}
where the right-hand side goes to zero as $i\rightarrow \infty$.
By the same argument, we also have $h_{L_i}\rightarrow 1$ in $C^{2,\alpha}(\mathbb{S}^{1})$.

Recall that $h\equiv 1$ is the unique solution to \eqref{eq:MP} with $f\equiv 1$. Observe that the linearized operator $L$ of \eqref{eq:MP} at $h\equiv 1$ is given by
\begin{equation}
Lv = \Delta_{\mathbb{S}^1} v + v+(1-p)v,
\end{equation}
but $2-p$ is not an eigenvalue of the Laplace--Beltrami operator $\Delta_{\mathbb{S}^1}$, which implies that the linearized operator $L$ is invertible. Hence, by the inverse function theorem, we conclude that $K_i=L_i$ for large $i$, which is a contradiction.
\end{proof}

\begin{proof} [Proof of \Cref{thm:MP}]
Let us prove the first assertion of the theorem. Assume to the contrary that for any $i \in \mathbb{N}$ there exists $f_i \in C^{\alpha}(\mathbb{S}^{1})$ with $\|f_i-1\|_{C^{\alpha}(\mathbb{S}^{1})} < 1/i$ such that there are two different solutions $K_i$ and $L_i$. By \Cref{lem:uniqueness}, at least one of them, say $K_i$, satisfies $\|h_{K_i}-1\|_{L^{\infty}(\mathbb{S}^{1})} > \varepsilon_0$. Moreover, by \Cref{thm:est} we have
\begin{equation}
\|h_{K_i}\|_{L^{\infty}(\mathbb{S}^{1})} \leq C.
\end{equation}
Therefore, the Blaschke's selection theorem shows that there exists a convex set $K_{\infty}$ such that $K_i \to K_{\infty}$ in Hausdorff distance as $i \to \infty$ up to subsequences. However, we know that $h_{K_{\infty}}$ satisfies \eqref{eq:MP} with $f=1$, which implies that $h_{K_{\infty}} \equiv 1$ by the uniqueness. This contradicts to the fact that $\|h_{K_{\infty}}-1\|_{L^{\infty}(\mathbb{S}^{1})} \geq \varepsilon_0$.

The second assertion of the theorem can be proved in a similar way. If it is not true, then for any $i \in \mathbb{N}$ we can find $f_i \in C^{\alpha}(\mathbb{S}^{1})$ with $\|f_i-1\|_{C^{\alpha}(\mathbb{S}^{1})} < 1/i$ such that the solution $K_i$ satisfies $\|h_{K_i}-1\|_{L^{\infty}(\mathbb{S}^{1})} \geq 1$, which yields a contradiction.
\end{proof}


\section*{Acknowledgement}


We want to thank Kyeongsu Choi for his interest in our work and valuable comments. Minhyun Kim gratefully acknowledges financial support by the German Research Foundation (GRK 2235 - 282638148). Taehun Lee has been supported by a KIAS Individual Grant (MG079501) at Korea Institute for Advanced Study.


\end{document}